\documentclass[pdftex, a4paper, 12pt, oneside]{amsart}
\usepackage{geometry}
\usepackage[cp1250]{inputenc}

\usepackage{alltt}
\usepackage{euler}
\usepackage{amsfonts}
\usepackage{amscd}
\usepackage{graphicx}
\usepackage{lpic}
\usepackage{longtable}
\usepackage{url}
\usepackage{import}
\usepackage{amssymb, amsthm, amsmath}
\usepackage[sc]{mathpazo}
\usepackage{enumitem, textcomp, setspace}

\usepackage{tikz}
\usetikzlibrary{intersections}
\usepackage{pgfplots}
\usepgfplotslibrary{external}

\allowdisplaybreaks

\numberwithin{equation}{section}

\theoremstyle{plain}
\newtheorem{theorem}{Theorem}[section]
\newtheorem{proposition}[theorem]{Proposition}

\newtheorem{corollary}[theorem]{Corollary}

\theoremstyle{definition}
\newtheorem{definition}[theorem]{Definition}
\newtheorem{example}[theorem]{Example}

\usepackage{booktabs}

\usepackage[pagebackref=false]{hyperref}

\usepackage[all]{hypcap}

\definecolor{newblue}{rgb}{0.27, 0.32, 0.86}
\definecolor{newred}{rgb}{0.86, 0.32, 0.27}
\hypersetup{
	pagebackref=true,
	colorlinks=true,       
	linkcolor=newred,          
	citecolor=newblue,        
	filecolor=magenta,      
	urlcolor=newblue           
}

\def\utr{\, \underline{\triangleright}\, }
\def\otr{\, \overline{\triangleright}\, }

\title[Biquandle invariant of immersed surface-links, Yoshikawa moves, ribbon 2-knots]{On biquandle-based invariant of immersed surface-links, Yoshikawa oriented fifth move, and ribbon 2-knots}

\author{Micha{\l} Jab{\l}onowski}

\address{Institute of Mathematics, Faculty of Mathematics, Physics and Informatics,\newline University of Gda\'nsk, 80-308 Gda\'nsk, Poland}



\email{michal.jablonowski@gmail.com}

\date{\today}

\graphicspath{ {fig18/} } 
\begin{document}

\maketitle

\begin{abstract}
We resolve an open problem by showing that the Yoshikawa's fifth oriented move in his list cannot be reproduced by any finite sequence of the other nine moves and planar isotopies. Our proof introduces a link-type semi-invariant that remains unchanged under all moves except the fifth, highlighting its necessity in generating the full move set.
Second, we extend the algebraic toolkit for immersed surface-links. After revisiting the banded-unlink description of immersed surfaces and the twelve local moves that relate their diagrams, we develop a biquandle-based coloring theory. By assigning elements of a biquandle to diagram arcs according to local rules, we obtain a counting invariant of immersed surfaces up to isotopy.
Third, we show that there are infinitely many pairs of ribbon $2$-knots with isomorphic groups but different knot quandles.
\end{abstract}

\renewcommand*{\arraystretch}{1.4}

\section{Introduction}\label{sec1}

The paper settles an open question about Yoshikawa's moves, the ten local moves that generate all diagrammatic isotopies of oriented surface-links. Previous work had proved the independence of other moves, not the oriented fifth move. We construct a semi-invariant derived from admissible marked-graph diagrams whose value is unchanged by every move in the generating set of other moves. Exhibiting two diagrams of the same surface-link that yield different values of the semi-invariant.
\par
Second, we enrich the algebraic framework for immersed (i.e., possibly self-intersecting) surface-links. We revisit the singular marked-graph (or "banded unlink") presentation, extending the catalog of oriented planar moves so that any two diagrams of equivalent immersed surfaces differ by the twelve types of moves. We equip an immersed surface-link with a biquandle structure. Each semi-arc of a singular marked diagram is colored by elements of a biquandle according to local rules that respect crossings, markers and singular points. Because the biquandle axioms correspond to the $\Gamma$-moves, the number of colorings is an invariant of the underlying surface type. The authors compute this new biquandle coloring invariant for an illustrative example, obtaining admissible colorings with a three-element biquandle.
\par
Third, we prove that there are infinitely many pairs of inequivalent ribbon $2$-knots that share the same fundamental group of the exterior but differ in the fundamental quandle.
\par 
This paper is organized as follows. In Section \ref{sec2}, we will prove an open problem of the independence of Yoshikawa oriented fifth move $\Gamma_5$ from the other generating set of oriented moves $\Gamma_1, \ldots, \Gamma_8, \Gamma_4', \Gamma_6'$ connecting diagrams of isotopic oriented surface-links. In Section \ref{sec3}, we review the singular marked graph diagram method of presenting immersed surface-links and present the oriented planar moves between these diagrams. In Section \ref{sec4}, we present the biquandle structure of an immersed surface-link with a method to obtain a biquandle coloring from a singular marked graph diagram and present a new invariant for oriented immersed surface-links. We show an example of the calculation. In Section \ref{sec5}, we show that there are infinitely many pairs of ribbon $2$-knots with isomorphic groups but different knot quandles.

\section*{Acknowledgements}

This research was funded in whole or in part by NCN 2023/07/X/ST1/00157.

\section{Independence of the Yoshikawa oriented fifth move}\label{sec2}

An embedding (or its image when no confusion arises) of a closed (i.e., compact, without boundary) surface $F$ into the Euclidean $\mathbb{R}^4$ (or into the $\mathbb{S}^4=\mathbb{R}^4\cup\{\infty\}$) is called a \emph{surface-link} (or \emph{surface-knot} if it is connected).
\par
Two surface-links are \emph{equivalent} if there exists an orientation-preserving homeomorphism of the four-space $\mathbb{R}^4$ to itself (or equivalently an auto-homeomorphism of the four-sphere $\mathbb{S}^4$), mapping one of those surfaces onto the other. See \cite{Kam17} for introductory material on surface-links.

A \emph{marked graph diagram} is a planar $4$-valent graph embedding, with the vertices decorated either by a classical crossing or marker decoration.

Any abstractly created marked graph diagram is an \emph{admissible diagram} if and only if both its resolutions are trivial classical link diagrams.

\par
It is known that the set of ten types of moves $\mathcal{G}=\{\Gamma_1, \ldots, \Gamma_8, \Gamma_4', \Gamma_6'\}$ (presented in Fig.\;\ref{michal_oriA} and \;\ref{michal_oriB0}), called oriented Yoshikawa moves (\cite{Yos94}), is a generating set of moves that relates two marked graph diagrams (modulo a planar isotopy) presenting equivalent oriented surface-links. In \cite{JKL13}, \cite{JKL15} it is shown that any Yoshikawa move from the set $\mathcal{G}\backslash\{\Gamma_5, \Gamma_8\}$ is independent of the other nine types. It is an open problem (see \cite{JKL15}, \cite{Oht16}) whether the move $\Gamma_5$ is independent of the other Yoshikawa moves from the set $\mathcal{G}$.

\begin{figure}[ht]
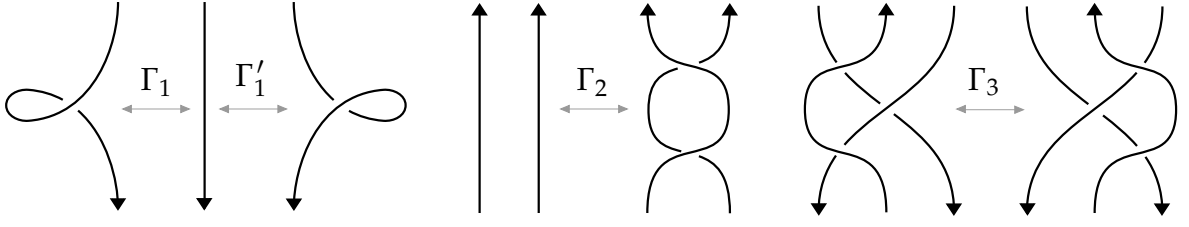

	\begin{center}
		\begin{lpic}[]{michal_oriA(15.5cm)}
			\lbl[b]{24,19;$\Gamma_1$}
			\lbl[b]{39,19;$\Gamma_1'$}
			\lbl[b]{93,19;$\Gamma_2$}
			\lbl[b]{155,19;$\Gamma_3$}
			
		\end{lpic}
		\caption{Moves, part I.\label{michal_oriA}}
	\end{center}
\end{figure}

\begin{figure}[ht]
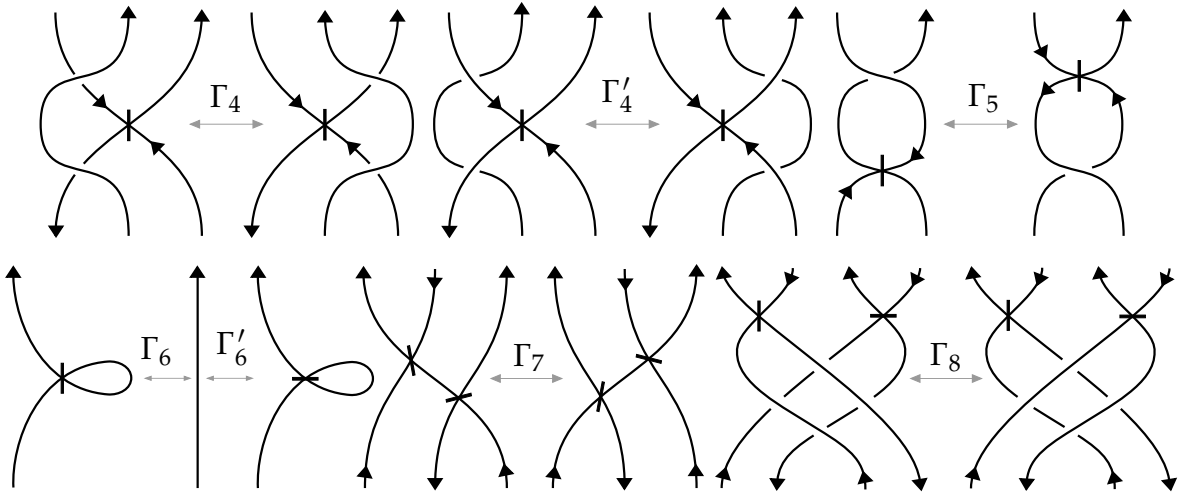

	\begin{center}
		\begin{lpic}[]{michal_oriB0(15.5cm)}
			\lbl[b]{35,60;$\Gamma_4$}
			\lbl[b]{97,60;$\Gamma_4'$}
			\lbl[b]{155,60;$\Gamma_5$}
			\lbl[b]{24,20;$\Gamma_6$}
			\lbl[b]{36,20;$\Gamma_6'$}
			\lbl[b]{83,19;$\Gamma_7$}
			\lbl[b]{149,19;$\Gamma_8$}
		\end{lpic}
		\caption{Moves, part II.\label{michal_oriB0}}
	\end{center}
\end{figure}

The independence of every Yoshikawa move will lead us to obtain a minimal generating set of moves on oriented graph diagrams. We will prove the following main theorem.

\begin{theorem}\label{thm1}
	
	The Yoshikawa move $\Gamma_5$ cannot be realized by a finite sequence of Yoshikawa moves of the other nine types from the set $\mathcal{G}$, and a planar isotopy.
	
\end{theorem}

\begin{proof}[Proof of theorem \ref{thm1}]
	We define a semi-invariant $L_*$ such that it preserves its values after performing each move from the set $\mathcal{G}\setminus \{\Gamma_5\}$, and we construct two pairs of admissible diagrams $D_1, D_2$ of equivalent surface-links such that $L_*(D_1)\not=L_*(D_2)$.
	\par 
	Define $L_*$ as an ambient isotopy class of classical links obtained from admissible diagrams by changing all oriented marker decorations as shown in Figure \ref{michal_080} respectively. It is straightforward to see that $L_*$ is unchanged by any move in $G\setminus \{\Gamma_5\}$.
	
	\begin{figure}[h]
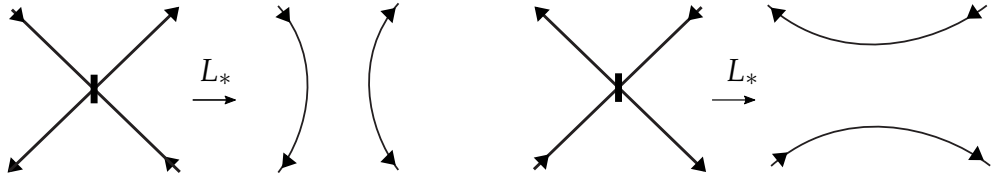

		\begin{center}
			\begin{lpic}[]{michal_080(13cm)}
				\lbl[b]{19,8;$L_*$}
				\lbl[b]{67,8;$L_*$}
			\end{lpic}
			\caption{Defining a semi-invariant $L_*$.\label{michal_080}}
		\end{center}
	\end{figure}
	
	The diagrams $D_1$ and $D_2$ represent equivalent surface-links, with the desired property explained above, are illustrated in Figure \ref{michal_030}, where we have that $L_*(D_1)$ is of the ambient isotopy type of the $6_1$ knot (i.e. a nontrivial classical knot with the Alexander polynomial $2-5t+2t^2$), on the other hand, $L_*(D_2)$ is ambient isotopy type of the trivial classical link.
	
	\begin{figure}[h]
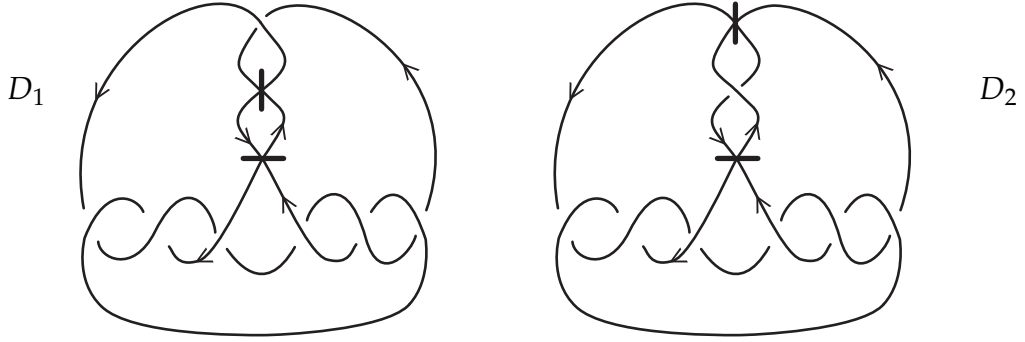

		\begin{center}
			\begin{lpic}[l(1.1cm),r(1.1cm)]{michal_030(11cm)}
				\lbl[r]{-2,15;$D_1$}
				\lbl[l]{55,15;$D_2$}
			\end{lpic}
			\caption{The diagrams $D_1$ and $D_2$.\label{michal_030}}
		\end{center}
	\end{figure}

\end{proof}

\section{Singular banded unlinks}\label{sec3}

	Let $X,Y$ be smooth ($C^{\infty}$) manifolds. Let $f:X^n\to Y^{m}$ be a smooth map. It is called an \emph{immersion} if at each point $x\in X$ the induced differential is a monomorphism. By the Whitney immersion theorem, any smooth map $f:X^n\to Y^{m}$ can be approximated homotopically with arbitrary accuracy by an immersion when $m\geq 2n$. 
	
	We consider smooth immersions $f:X^n\to Y^{2n}$ such that the following three conditions are satisfied: 
	
	(i) $\#|f^{-1}(f(x))|\leq 2$,
	(ii) there are only a finite number of points with $\#|f^{-1}(f(x))|= 2$,
	(iii) at each singularity $p=f(x)=f(y)$, there is a coordinate chart around $p$ where the two coordinate subspaces $\mathbb{R}^n\times 0$ and $0 \times\mathbb{R}^n$ are exactly the immersed images of $f$ near $x$ and $y$ respectively. That is the map is "self-transverse". \\
	
	By general position theorems for maps, any smooth map $f:X^n\to Y^{2n}$ can be approximated homotopically with arbitrary accuracy by an immersion described above.
	
	When we consider the case of immersions $\mathbb{S}^1\to\mathbb{R}^2$ we have the "classical" case, which includes, for example, planar projections of knots and links embedded in $\mathbb{R}^3$, losing the (codimension two) knotting phenomena. The next-dimension case is $f:X^2\to \mathbb{R}^{4}$ when we have both the finite point generic intersections and nontrivial generic embeddings (a known generalization of the classical knot theory).
	
	An immersion in this dimensions (or its image when no confusion arises) of a closed (i.e. compact, without boundary) surface $F$ into the Euclidean $\mathbb{R}^4$ (or into the $\mathbb{S}^4=\mathbb{R}^4\cup\{\infty\}$) is called an \emph{immersed surface-link} (or \emph{immersed surface-knot} if it is connected).
	\par
	Two immersed surface-links are \emph{equivalent} if there exists an orientation-preserving homeomorphism of the four-space $\mathbb{R}^4$ to itself (or equivalently an auto-homeomorphism of the four-sphere $\mathbb{S}^4$), mapping one of those surfaces onto the other. Fix an immersed surface-link $F$ embedded in a manifold $\mathbb{S}^4$. For an open neighborhood, denoted $N(F)$, the \emph{exterior} of $F$ is $E(F):=\mathbb{S}^4\backslash N(F)$. 
	\par 
	If two singular surface-links are equivalent, then their exteriors are diffeomorphic.

		A \emph{singular link} $L$ in $\mathbb{R}^3$ is the image of an immersion in the classical case $\iota: S^1 \sqcup \cdots \sqcup S^1 \rightarrow \mathbb{R}^3$ which is injective except at isolated double points that are not tangencies. At every double point $p$ we include a small disk $v\cong D^2$ embedded in $\mathbb{R}^3$.  We refer to these disks as the {\emph{vertices}} of $L$. The double points of a singular link $L$ correspond to the isolated double points of an immersed surface in $\mathbb{R}^4$.
		
		For each vertex $v$ of $L$, these two opposite push-offs form a bigon in a neighborhood of $v$, which bounds an embedded disk $D_v$.  This disk $D_v$ can be chosen so that its interior intersects $L$ transversely in a single point near $v$.  For each vertex $v$ select such a disk $D_v$ (ensuring that all of these disks are pairwise disjoint).
		
		\begin{figure}
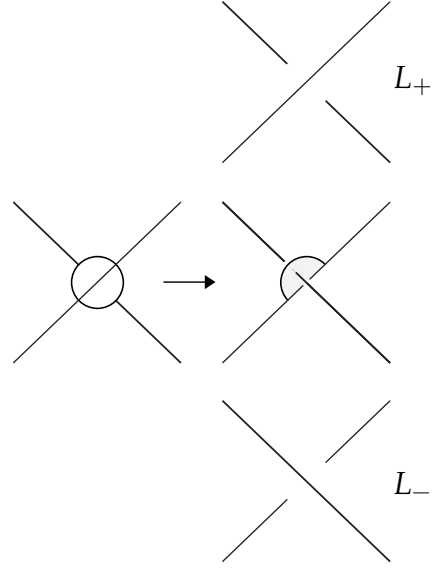

			\begin{center}
				\begin{lpic}[r(3cm)]{michal_11(5cm)}
					\lbl[l]{75,95;$L_+$}
					\lbl[l]{75,15;$L_-$}
				\end{lpic}
			\end{center}
			\caption{A vertex of a marked singular link and the corresponding surface cuts.}\label{michal_11}
			
		\end{figure}

		Let $D_L$ denote the union of all of these embedded disks. Let $(L, \sigma, B)$ be an singular link with bands $B=\{b_1, \dots, b_n\}$ and $\Delta_1,\ldots,\Delta_a \subset \mathbb{R}^3$ be mutually disjoint $2$-disks with $\partial(\cup_{j=1}^a\Delta_j)= L_{B+}$, and let $\Delta_1',\ldots,\Delta_b' \subset \mathbb{R}^3$ be mutually disjoint $2$-disks with $\partial(\cup_{k=1}^b\Delta_k')= L_-$.
		
		We define $\Sigma \subset \mathbb{R}^3 \times \mathbb{R} = \mathbb{R}^4 $ an \emph{immersed surface-link corresponding to} $(L, \sigma, B)$ by the following cross-sections.
		
		$$
		(\mathbb{R}^3_t, \Sigma \cap \mathbb{R}^3_t)=\left\{%
		\begin{array}{ll}
			(\mathbb R^3, \emptyset) & \hbox{for $t > 1$,}\\
			(\mathbb R^3, L_{B+} \cup (\cup_{j=1}^a\Delta_j)) & \hbox{for $t = 1$,} \\
			(\mathbb R^3, L_{B+}) & \hbox{for $0 < t < 1$,} \\
			(\mathbb R^3, L_+\cup (\cup_{i=1}^n b_i)) & \hbox{for $t = 0$,} \\
			(\mathbb R^3, L_+) & \hbox{for $-1/2 < t < 0$,} \\
			(\mathbb R^3, L_- \cup D_L) & \hbox{for $t = -1/2$,} \\
			(\mathbb R^3, L_-) & \hbox{for $-1 < t < -1/2$,} \\
			(\mathbb R^3, L_- \cup (\cup_{k=1}^b\Delta_k')) & \hbox{for $t = -1$,} \\
			(\mathbb R^3, \emptyset) & \hbox{for $ t < -1$.} \\
		\end{array}
		\right.
		$$
		
		By an ambient isotopy of $\mathbb R^3$, we shorten the bands of a singular link with bands $LB$ so that each band is contained in a small $2$-disk. Replacing the neighborhood of each band with the neighborhood of a $4$-valent marked vertex as in Fig.\;\ref{MJ_100}, we obtain a singular marked graph.
		
		\begin{figure}[ht]
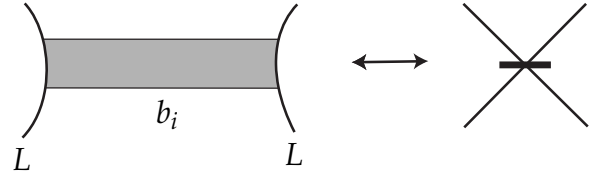

			\begin{center}
				\begin{lpic}[b(0.5cm)]{MJ_100(7.5cm)}
					\lbl[t]{30,8;$b_i$}
					\lbl[t]{0,-2;$L$}
					\lbl[t]{56,-1;$L$}
				\end{lpic}
				\caption{A band corresponding to a marked vertex.\label{MJ_100}}
			\end{center}
		\end{figure}

		A \emph{singular marked graph diagram} is a planar $4$-valent graph embedding, with the vertices decorated either by a classical crossing, marker or a singular decoration.
		
		\begin{figure}
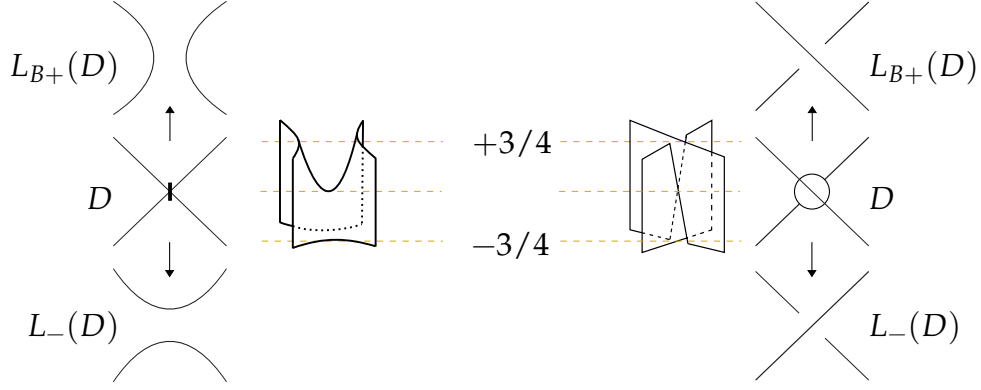

			\begin{center}
				\begin{lpic}[]{michal_06(10cm)}
					\lbl[l]{105,40;$-3/4$}
					\lbl[l]{105,70;$+3/4$}
					\lbl[r]{2,16;$L_-(D)$}
					\lbl[r]{0,54;$D$}
					\lbl[r]{2,92;$L_{B+}(D)$}
					\lbl[l]{221,16;$L_-(D)$}
					\lbl[l]{221,54;$D$}
					\lbl[l]{221,92;$L_{B+}(D)$}
				\end{lpic}
				\caption{A neighborhood in the four-space of a marker and a singular point.\label{michal_06}}
			\end{center}
		\end{figure}

		\begin{figure}
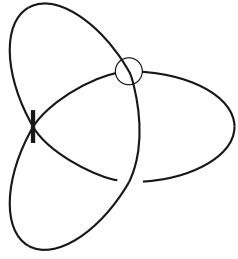

			
			\begin{center}
				\begin{lpic}[]{michal_14(3cm)}
					
				\end{lpic}
				\caption{Example of an admissible singular marked graph diagram.\label{michal_14}}
			\end{center}
			
		\end{figure}
		
		Any abstractly created singular marked graph diagram is an \emph{admissible diagram} if and only if both its resolutions $L_-$ and $L_{B+}$ are trivial classical link diagrams.

\begin{theorem}[\cite{HKM25, Jab23}]
	Two oriented immersed surface-links $\mathcal{L}_1$ and $\mathcal{L}_2$  are equivalent if and only if their singular marked diagrams are related by a finite sequence of moves $\Gamma_1, \ldots, \Gamma_{12}d$ shown in Figures \ref{michal_oriA}--\ref{michal_oriB0},  \ref{michal_oriC}--\ref{michal_oriE}.
\end{theorem}

\begin{figure}[ht]
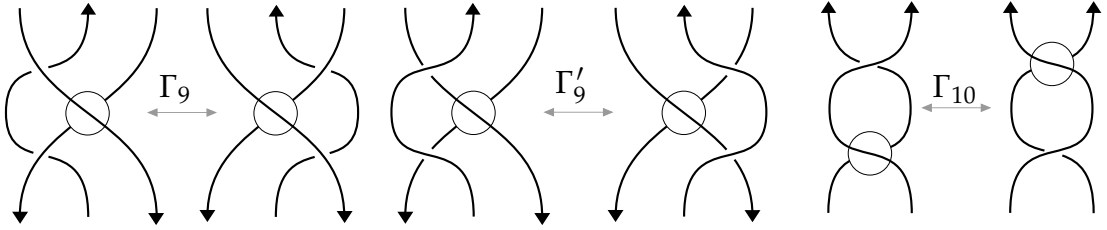

	\begin{center}
		\begin{lpic}[]{michal_oriC(14.5cm)}
			\lbl[b]{28,20;$\Gamma_9$}
			\lbl[b]{93,20;$\Gamma_9'$}
			\lbl[b]{156,20;$\Gamma_{10}$}
		\end{lpic}
		\caption{Moves, part III.\label{michal_oriC}}
	\end{center}
\end{figure}

\begin{figure}[ht]
	\begin{center}
		\begin{lpic}[]{michal_oriD(11.5cm)}
			\lbl[b]{35,62;$\Gamma_{11}a$}
			\lbl[b]{112,62;$\Gamma_{11}b$}
			\lbl[b]{35,20;$\Gamma_{11}c$}
			\lbl[b]{112,20;$\Gamma_{11}d$}
		\end{lpic}
		\caption{Moves, part IV.\label{michal_oriD}}
	\end{center}
\end{figure}

\begin{figure}[ht]
	\begin{center}
		\begin{lpic}[]{michal_oriE(11.5cm)}
			\lbl[b]{30,62;$\Gamma_{12}a$}
			\lbl[b]{98,62;$\Gamma_{12}b$}
			\lbl[b]{30,20;$\Gamma_{12}c$}
			\lbl[b]{98,20;$\Gamma_{12}d$}
		\end{lpic}
		\caption{Moves, part V.\label{michal_oriE}}
	\end{center}
\end{figure}

\section{Biquandle structure}\label{sec4}

A generalization of quandles (called biquandles) was introduced in \cite{KR03}. A \emph{biquandle} is an algebraic structure with two binary operations satisfying certain conditions which can be presented by semi-arcs of links (or semi-sheets of surface-links) as its generators modulo oriented Reidemeister moves (or Roseman moves). In \cite{CES04}, J. S. Carter, M. Elhamdadi, and M. Saito introduced and used cocycles to define invariants via colorings of link diagrams by biquandles and a state-sum formulation.
\par 
In \cite{KKKL18} S. Kamada, A. Kawauchi, J. Kim and S. Y. Lee discussed the (co)homology theory of biquandles and developed the biquandle cocycle invariants for oriented surface-links by using broken surface diagrams generalizing quandle cocycle invariants. Then showed how to compute the biquandle cocycle invariants from marked graph diagrams.

\begin{definition}
	Let $X$ be a set. A \textit{biquandle structure} on $X$ is a
	pair of maps $\utr,\otr:X\times X\to X$ satisfying:
	\begin{itemize}
		\item[(i)] for all $x\in X$, $x\utr x=x\otr x$,
		\item[(ii)] the maps $\alpha_y,\beta_y:X\to X$ for all $y\in X$ 
		and $S:X\times X\to X\times X$ defined by
		\[\alpha_y(x)=x\otr y,\quad \beta_y(x)=x\utr y\quad \mathrm{and}\quad
		S(x,y)=(y\otr x, x\utr y)\]
		are invertible, and
		\item[(iii)] for all $x,y,z\in X$ we have the \textit{exchange laws}:

		\begin{enumerate}
			\item 			$(x\utr y)\utr (z\utr y) =  (x\utr z)\utr(y\otr z),$
			\item 			$(x\otr y)\utr (z\otr y) =  (x\utr z)\otr(y\utr z),$
			\item 			$(x\otr y)\otr (z\otr y) =  (x\otr z)\otr(y\utr z).$
		\end{enumerate}

	\end{itemize}
	Axiom (ii) is equivalent to the \emph{adjacent labels rule},
	which says that in the ordered quadruple $(x,y,x\utr y,y\otr x)$, any two 
	neighboring entries (including $(y\otr x, x)$ determine the other two.
	A \emph{biquandle} is a set $X$ with a choice of biquandle structure.
\end{definition}

\begin{example}
	A $\mathbb{Z}[t^{\pm 1},s^{\pm 1}]$-module has a biquandle structure known as
	an \emph{Alexander biquandle} defined by
	\[x\utr y = tx+(s-t)y \quad\mathrm{and}\quad  x\otr y = sx.\]
	In particular, a choice of units $t,s\in\mathbb{Z}_n$ defines an
	Alexander biquandle structure on $\mathbb{Z}_n$.
\end{example}

\begin{figure}[ht]
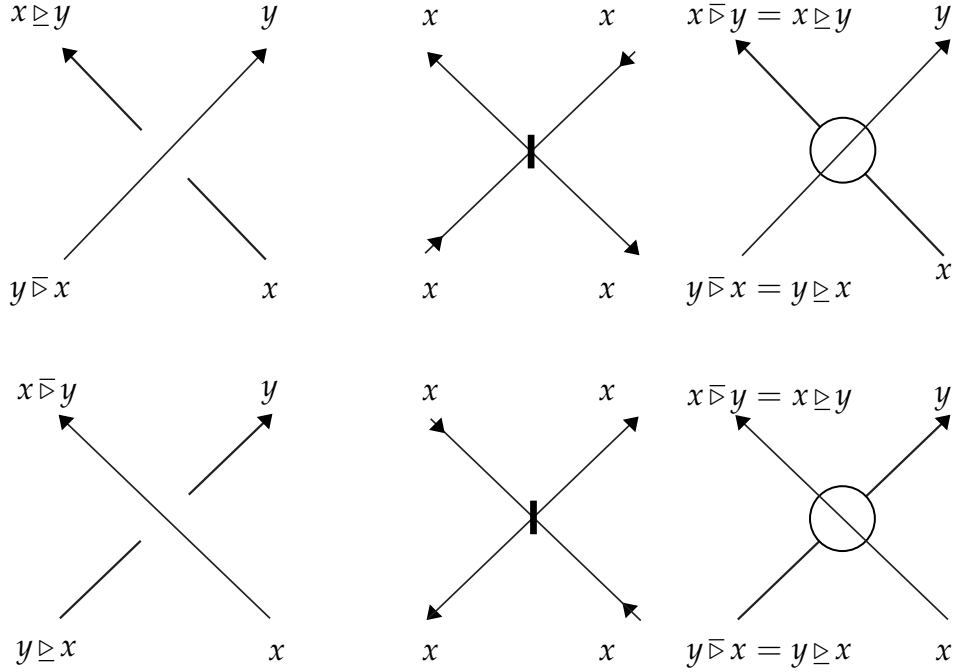

	\begin{center}
		\begin{lpic}[l(0.8cm),b(1.0cm),r(0.4cm),t(0.5cm)]{michal_biq_2(11.8cm)}
			\lbl[r]{2,95;$x\utr y$}
			\lbl[l]{32,95;$y$}
			\lbl[r]{60,95;$x$}
			\lbl[l]{85,95;$x$}
			\lbl[r]{60,36;$x$}
			\lbl[l]{85,36;$x$}
			\lbl[r]{2,52;$y\otr x$}
			\lbl[l]{32,52;$x$}
			\lbl[r]{60,52;$x$}
			\lbl[l]{85,52;$x$}
			\lbl[r]{60,-5;$x$}
			\lbl[l]{85,-5;$x$}
			\lbl[r]{3,36;$x \otr y$}
			\lbl[l]{32,36;$y$}
			\lbl[r]{3,-5;$y \utr x$}
			\lbl[l]{33,-5;$x$}
			
			\lbl[r]{125,95;$x\otr y = x\utr y$}
			\lbl[r]{125,52;$y\otr x = y\utr x$}
			\lbl[l]{138,55;$x$}
			\lbl[l]{138,95;$y$}

			\lbl[r]{125,35;$x\otr y = x\utr y$}
			\lbl[r]{125,-5;$y\otr x = y\utr x$}
			\lbl[l]{138,-5;$x$}
			\lbl[l]{138,35;$y$}

		\end{lpic}
		\caption{Labeling rules at crossings, markers, and singular points.\label{michal_biq2}}
	\end{center}
\end{figure}

Let $X$ and $Y$ be biquandles. A function $f:X\rightarrow Y$ is called a {\it biquandle homomorphism} if $f({x}\utr{y})={f(x)}\utr {f(y)}$ and $f({x}\otr {y})={f(x)}\otr {f(y)}$ for any $x,y\in X.$ We denote the set of all biquandle isomorphic from $X$ to $Y$ by ${\rm Hom}(X,Y)$. A bijective biquandle homomorphism is called a \emph{biquandle isomorphism}. Two biquandles $X$ and $Y$ are said to be {\it isomorphic} if there is a biquandle isomorphism $f: X \to Y$. 

Let $D$ be a singular marked diagram and let $C(D)$, $V(D)$ and $S(D)$ denote the set of all crossings, marked vertices and singular points of $D$, respectively. By a {\it semi-arc} of $D$ we mean a connected component of $D \setminus (C(D) \cup V(D) \cup S(D))$.

\begin{definition}

Let us fix a biquandle $X$, let $D$ be an oriented singular marked diagram of an oriented immersed surface-link, and let $\Lambda$ be the set of components of this diagram. A \emph{biquandle coloring} ${\mathcal C}$ is a mapping ${\mathcal C}:\Lambda\to X$ such that around each classical, marked or singular point, the relation shown in Figure \ref{michal_biq2} holds. These conditions are consistent around each classical, marked or singular point due to the axioms for the biquandle. Denote by ${Col}_X(D)$ the set of all colorings of the diagram $D$ with biquandle $X$.

\end{definition}

\begin{proposition}
	The biquandle axioms are chosen such that given a biquandle coloring of one side of any $\Gamma$--type move, there is a unique biquandle coloring of the other side of the move with the condition that colors agree on the boundary arcs that leave the disc where the move is performed. 
\end{proposition}

\begin{proof}
	The proof for the moves $\Gamma_1, \ldots, \Gamma_3$ can be found in the existing literature (see \cite{KNS16}), the axioms for the biquandle were motivated to satisfy those Reidemeister moves, for moves $\Gamma_4, \ldots, \Gamma_8$ see for example \cite{JN20}. We checked, by standard hand calculations, the proof for the remaining moves, using the diagrams shown in Figures \ref{michal_oriA}--\ref{michal_oriB0},  \ref{michal_oriC}--\ref{michal_oriE}.
	\par 
	In particular, checking the validity for the moves $\Gamma_9$ and $\Gamma_9'$ is similar to the case of $\Gamma_3$ case. The validity for the moves $\Gamma_{12}a-\Gamma_{12}d$ and $\Gamma_{11}a-\Gamma_{11}d$ are similar to the $\Gamma_4$ and $\Gamma_4'$ cases, because of the fact that $x\otr y = x\utr y$ and $y\otr x = y\utr x$ around the singular point in Figure \ref{michal_biq2}.

\end{proof}

Ashihara \cite{Ash12} gave a method to calculate the fundamental biquandle $BQ(L)$ of an oriented surface-link from its marked graph diagram. To get a fundamental biquandle $BQ_S(L)$ of an oriented singular surface-link we can take the analogous presentation $\left< S\;|\; R\right>$ as the quotient of the free biquandle on the set $S$ (generators associated to the semiarcs of the diagram) by the equivalence relation generated by relations $R$ (see Figure \ref{michal_biq2}) from any diagram for $L$ (see \cite{KKKL18} for more details of the construction). From the previous Proposition the biquandles $BQ_S(L_1)$ and $BQ_S(L_2)$ are isomorphic if $L_1$ and $L_2$ are equivalent singular surface-links.

\begin{corollary}
	The number of biquandle colorings $\#Col_X(D)$ of an oriented singular marked diagram $D$ is an invariant of an oriented singular surface-link $\mathcal{L}$ presented by $D$, we can denote it therefore by $\#Col_X(\mathcal{L})$ and also called biquandle coloring invariant.
\end{corollary}

Let us consider the singular surface link $L_T$, shown as the singular marked diagram $D_T$ in Figure \ref{michal_031}. Let a biquandle on $X_T=\{1, 2, 3\}$ be the biquandle with the operation given by the following matrix.
\renewcommand*{\arraystretch}{1.4}	
	
	\[
	\begin{array}{r|rrr} 
		\utr & 1 & 2 & 3\\\hline
		1 & 3 & 1 & 3\\
		2 & 2 & 2 & 2\\
		3 &	1 & 3 & 1
	\end{array}\quad
	\begin{array}{r|rrr}
		\otr & 1 & 2 & 3\\\hline
1 & 3 & 3 & 3\\
2 & 2 & 2 & 2\\
3 &	1 & 1 & 1
	\end{array}
	\]

We find that $\#Col_{X_T}(D_T)= 5$, two of the admissible coloring is shown in Figure \ref{michal_031}, the other two admissible colorings are obtained from the previous by exchanging colors $1$ with $3$, and the fifth one is monochromatic, colored by element $2$.

	\begin{figure}[h]
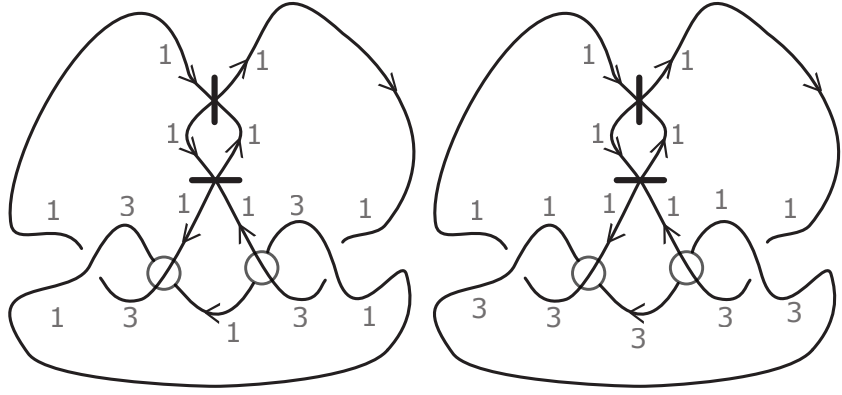

	\begin{center}
		\begin{lpic}[]{michal_032(11cm)}
		\end{lpic}
		\caption{Admissible colorings of the diagram $D_T$.\label{michal_031}}
	\end{center}
\end{figure}

\section{Ribbon $2$-knots and their fundamental quandles}\label{sec5}

Let us consider an oriented surface-knot $F$ in $\mathbb{S}^4$. From Tubular Neighborhood Theorem $F$ admits a tubular neighborhood, denoted $N(F)$, i.e., an image of a smooth embedding $N:\mathbb{D}^2\times F$ such that for any $p\in F$ we have $N(0,p)=p$. The \emph{exterior} of $F$ is $E(F):=\mathbb{S}^4\backslash int(N(F))$. If two surface-knots are equivalent, then their exteriors are diffeomorphic. The \emph{group} of a surface-link $F$ denoted $\mathcal{G}(F)$ is $\pi_1(E(F))$, the group is finitely presented.
\par The other invariant of surface-knot $F$ is its fundamental quandle $\mathcal{Q}(F)$, associated group of the knot quandle is group isomorphic to the knot group (see \cite{Kam17} for an introduction). Every $2$-knot that has a ($\mathbb{R}^3$) diagram without triple points is ribbon \cite{Yaj64}. Two ribbon $2$-knots or ribbon torus-knots with isomorphic fundamental quandles have isomorphic fundamental biquandles \cite{Ash14}. Therefore, we focus in this section on fundamental quandles, for this family of $2$-knots. It is known that there exist arbitrarily many inequivalent surface-knots with the same knot fundamental quandle \cite{Tan07}.
\par 
Recently, in \cite{TT25} it is shown that there exist infinitely many pairs of $2$-knots not distinguishable by a knot group but tell apart by their quandles. However, none of those pairs were both ribbon $2$-knots, at least one of them was non-ribbon twist-spun knot. It is natural to ask a question if there is a pair of ribbon $2$-knots with isomorphic groups that have different knot quandles. We obtain the following.

\begin{theorem}
There are infinitely many pairs of ribbon $2$-knots with mutually isomorphic fundamental groups of their exteriors but mutually non-isomorphic knot fundamental quandles.
\end{theorem}

\begin{proof}
	
	Suciu \cite{Suc85} constructed an infinite family of mutually non-ambient-isotopic $2$-knots (with non-isomorphic $\pi_2$ as $\mathbb{Z}\pi_1$ modules), denoted $R_k$, for any integer $k>0$. We show that this family of surface-knots has the desired property in the Theorem. Each $\mathcal{G}(R_k)$ is isomorphic to the trefoil classical knot group $\left<t, x\;|\;txt=xtx \right>$.
	\par 
	First, let us derive the presentation for the fundamental biquandle $\mathcal{Q}(R_k)$. The presentation for the fundamental biquandle can be calculated directly from the marked graph diagram for surface-knot \cite{Ash12}. It agrees with the method described earlier in this paper. Specifically, we consider relations as in Figure \ref{michal_biq2}, but without singular points and with the trivial one of the operations, namely $x\otr y = x$ for all $x, y$.
	\par 
	We present a marked graph diagram for $R_k$ in Figure \ref*{sc2_4_quandle}, which is the modified, by the relation in Figure \ref{MJ_100}, version of the banded unlink diagram for $R_k$ presented in \cite[p. 55]{Rup22}. In the Figure, the polygonal green region has to appear $k-1$ times to obtain $R_k$. From the arc relations, we see that if we put arc generators $c, d, t$, then we end up with two relations, marked by the red rectangles.

	\begin{figure}[h!t]
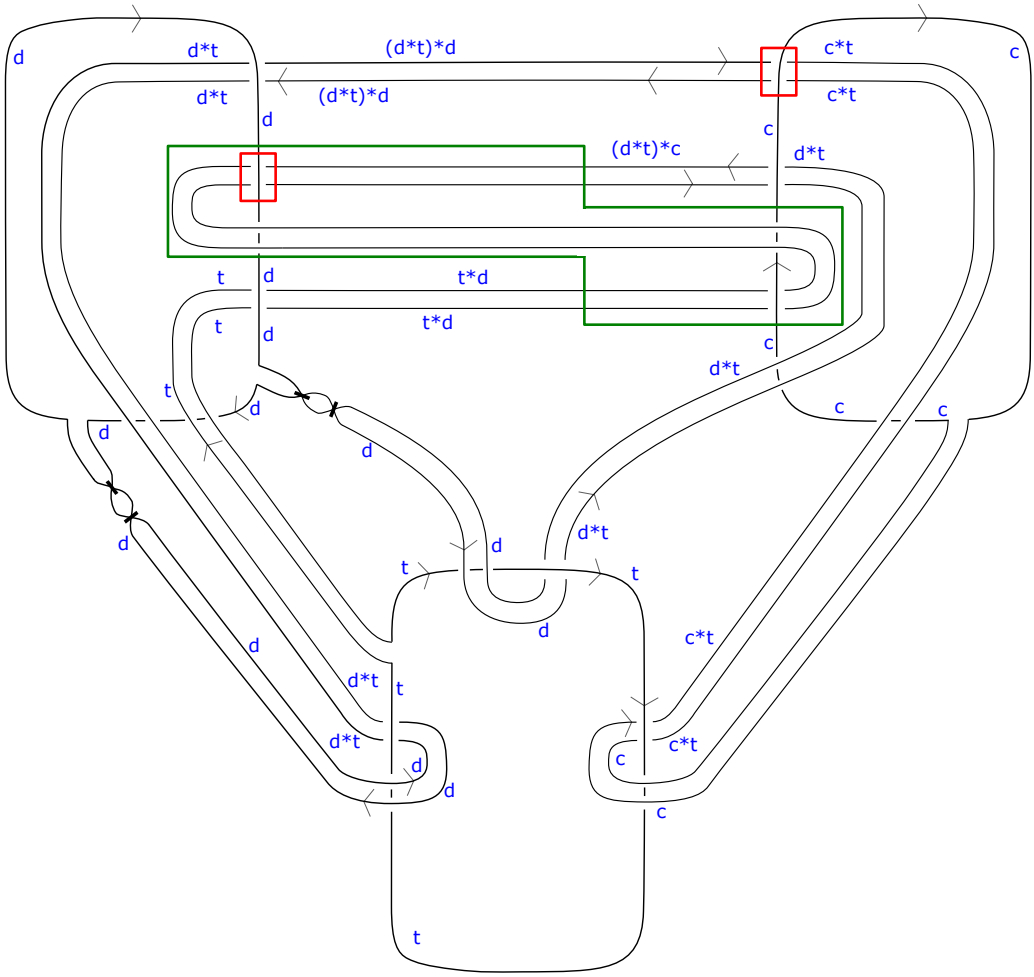

		\begin{center}
			\begin{lpic}[]{sc2_5_quandle(14.5cm)}
			\end{lpic}
			\caption{A marked graph diagrams for $R_k$.\label{sc2_4_quandle}}
		\end{center}
	\end{figure}

	Therefore, for any integer $k>0$:
	$$\mathcal{Q}(R_k)=\left<c, d, t\;|\; (d*t)*d=(c*t)*c\;,\; (d*t)*c=(\ldots(t*d)\overline{*}c)*d\ldots )\overline{*}c)*d\right>,$$
	
	where $\overline{*}$ is the dual quandle operation to $*$, and in the second dots $\ldots$ symbols "$)\overline{*}c)*d$" appear $k-2$ times.
	
	\par 
	It is of no use here to consider the quandle cocycle invariant (see \cite{Kam17} for the definition) to distinguish any pair of them. This is because the quandle cocycle invariant can be calculated from the neighborhoods of the triple points in the projection, and each ribbon $2$-knot has a projection without any triple point (as we stated earlier).
\par 
We know that $\mathcal{G}(R_i)\cong \mathcal{G}(R_j)$ for any $i\not= j$. Note that it is also of no use here to consider the \emph{associated group} $As(Q)=\left<x\in Q\;|\;x*y=y^{-1}xy\; (\text{for } x, y \in Q)\right>$, because $As(\mathcal{Q}(R_i))\cong \mathcal{G}(R_i)\cong \mathcal{G}(R_j) \cong As(\mathcal{Q}(R_i))$ for any $i\not= j$.
Our idea to show that $\mathcal{Q}(R_i)\not\cong \mathcal{Q}(R_j)$ for any $i\not= j$ is to use the core groups of their fundamental quandles. The \emph{core group} of the quandle $Q$ is a group with presentation $Core(Q)=\left<x\in Q\;|\;x*y=yx^{-1}y\; (\text{for } x, y \in Q)\right>$.
\par
We can straightforward calculate their presentations (for any integer $k>0$) with substitutions ($t=t^{-1}$, $w=dc^{-1}$):
$$Core(\mathcal{Q}(R_k))=\left<c, t, w\;|\; wctwctw=ctct \;;\; ctwct=w^{k+1}ctw^{k+1} \right>.$$

To show that these groups are mutually non-isomorphic, we use pure group theory \cite[chapter 5]{MKS04}, \cite[chapter 5.1]{Rob96}.

Denote $A_k = Core(\mathcal{Q}(R_k))$ for short, we want to show that $A_n\not\cong A_m$ for any $n\not = m$. In the abelianization $A_k/[A_k,A_k]\cong \mathbb{Z}\oplus\mathbb{Z}_3$ the two defining relations becomes $3[w]=0$, $[c]+[t]=(2k+1)[w]$. The order of $[c]+[t]$ in the torsion summand is $3/gcd(3, 2k+1)$.

A direct Hall-Witt computation gives
$$
[w^{3},c]=[[w,c],w]^{3}\ [[w,c],w,c],\qquad
[w^{3},t]=[[w,t],w]^{3}\ [[w,t],w,t],
$$
so $w^{3}$ is central whenever every weight-2 commutator with $w$ has exponent dividing $3$.

Therefore $w^3$ is central iff $k\equiv 1\pmod 3$. So if $n\not\equiv m\pmod 3$ then $A_n\not\cong A_m$. 
\par 
Assume that $n\equiv m\pmod 3$. Consider group lower-central series $\gamma_1(G)=G, \gamma_{s+1}(G)=[\gamma_s(G),G]$. Assume that (in the contrary) $A_n\cong A_m$, then because $\gamma_s(G)$ is characteristic we have $A_n/\gamma_4(A_n)\cong A_m/\gamma_4(A_m)$. With the standard Hall commutator calculus, only the triple commutator $[[c,t],w]$ survives. 

\par 
Let $y=[[c,t],w]$.
Using the classical Hall basis for a free group on three generators and the two defining relators, one verifies that
$
\gamma_{3}(A_k)/\gamma_{4}(A_k)
=\langle y\mid y^{3}=1,\;y^{\,2k+1}=1\rangle
\cong \mathbb Z_3/\langle 2k+1\rangle .
$
Hence, the exponent of this quotient is $3/\gcd(3,2k+1)$.  This immediately forces
$
A_n\cong A_m\;\Longrightarrow\; n\equiv m\pmod9,
$
We pass to $A_k/\gamma_5(A_k)$, then to $A_k/\gamma_6(A_k)$, etc. Because $n\not = m$, their expansions in base $3$ as integers differ at some digit. Therefore, eventually (after finitely many steps) the quotients become non-isomorphic, a contradiction with $A_n\cong A_m$.

\end{proof}


\begin{thebibliography}{99}

\bibitem[Ash12]{Ash12} S. Ashihara, Calculating the fundamental biquandles of surface links from their ch-diagrams, \emph{J. Knot Theory Ramifications} 21 (2012), 1250102.

\bibitem[Ash14]{Ash14} S. Ashihara, Fundamental biquandles of ribbon $2$-knots and ribbon torus-knots with isomorphic fundamental quandles, \emph{Journal of Knot Theory and Its Ramifications} 23, no. 01 (2014), 1450001.

\bibitem[CES04]{CES04} J.S. Carter, M. Elhamdadi, and M. Saito, Homology Theory for the Set-Theoretic Yang-Baxter Equation and Knot Invariants from Generalizations of Quandles, \emph{Fund. Math.} 184 (2004), 31--54.

\bibitem[HKM25]{HKM25} M.~Hughes, S.~Kim, and M.~Miller, Band diagrams of immersed surfaces in 4-manifolds, \emph{Algebr. Geom. Topol.} 25:3 (2025), 1731--1791.

\bibitem[Jab23]{Jab23} M. Jab\l onowski, Minimal generating sets of moves for surfaces immersed in the four-space, \emph{J. Knot Theory Ramifications} 32 (2023), 2350071.

\bibitem[JKL13]{JKL13} Y. Joung, J. Kim, and S. Y. Lee, Ideal coset invariants for surface-links in $\mathbb{R}^4$, \emph{J. Knot Theory Ramifications} 22 (2013), 1350052.

\bibitem[JKL15]{JKL15} Y. Joung, J. Kim, and S.Y. Lee, On generating sets of Yoshikawa moves for marked graph diagrams of surface-links, \emph{J. Knot Theory Ramifications} 24 (2015), 1550018.

\bibitem[JN20]{JN20} Y. Joung and S. Nelson, Biquandle module invariants of oriented surface-links, \emph{Proceedings of the American Mathematical Society} 148.7 (2020), 3135--3148.

\bibitem[KNS16]{KNS16} A. Kaestner, S. Nelson, and L. Selker, Parity biquandle invariants of virtual knots,
\emph{Topology and its Applications} 209 (2016), 207--219.

\bibitem[Kam17]{Kam17} S. Kamada, \emph{Surface-Knots in $4$-Space}, Springer Monographs in Mathematics, Springer (2017).

\bibitem[KKKL18]{KKKL18} S. Kamada, A. Kawauchi, J. Kim, and S.Y. Lee, Biquandle cohomology and state-sum invariants of links and surface-links. \emph{Journal of Knot Theory and Its Ramifications} 27(11) (2018), 1843016.

\bibitem[KR03]{KR03} L.H. Kauffman and D.E. Radford, Bi-oriented quantum algebras, and generalized Alexander polynomial for virtual links, \emph{Contemp. Math.} 318 (2003), 113--140.

\bibitem[MKS04]{MKS04} W. Magnus, A. Karrass, and D. Solitar, \emph{Combinatorial Group Theory}, 2nd rev. ed., Dover (2004).

\bibitem[Oht16]{Oht16} T. Ohtsuki, "Problems on Low-dimensional Topology, 2016 (Intelligence of Low-dimensional Topology)." (2016), 115-129.

\bibitem[Rob96]{Rob96} D.J.S. Robinson, \emph{A Course in the Theory of Groups}, 2nd ed., GTM 80, Springer (1996).

\bibitem[Rup22]{Rup22} B.M. Ruppik, PhD thesis: \emph{Casson-Whitney unknotting, Deep slice knots and Group trisections of knotted surface type}, Rheinische Friedrich-Wilhelms-Universität Bonn (2022).

\bibitem[Suc85]{Suc85} A.I. Suciu, Infinitely many ribbon knots with the same fundamental group, \emph{Mathematical Proceedings of the Cambridge Philosophical Society} 98.3 (1985), 481--492.

\bibitem[Tan07]{Tan07} K. Tanaka, Inequivalent surface-knots with the same knot quandle, \emph{Topology Appl.} 154 (2007), 2757--2763.

\bibitem[TT25]{TT25} K. Tanaka and Y. Taniguchi, $2$-knots with the same knot group but different knot quandles, to appear in \emph{J. Math. Soc. Japan} (2025).

\bibitem[Yaj64]{Yaj64} T. Yajima, On simply knotted spheres in $\mathbb{R}^4$, \emph{Osaka J. Math.} 1(2) (1964), 133--152.

\bibitem[Yos94]{Yos94} K.~Yoshikawa, An enumeration of surfaces in four-space, \emph{Osaka J. Math.} 31 (1994), 497--522.
	
\end{thebibliography}
\end{document}